\documentclass[a4paper,10pt]{extarticle}
\usepackage[british,american]{babel}
      
\usepackage{amsmath}
\usepackage{amsthm,verbatim}
\usepackage{amssymb}
\usepackage{mathtools}
\usepackage{color,graphicx}  
\usepackage{enumitem}   
\usepackage{pstricks-add}
\usepackage{todonotes}
\usepackage{pgf,tikz,pgfplots}
\usetikzlibrary{decorations.pathmorphing}
\usepackage{physics}
\usetikzlibrary{patterns}
\usetikzlibrary{intersections}
\usetikzlibrary{svg.path}
\usetikzlibrary{hobby,decorations.markings}
\usepackage{mathrsfs}
\usepackage{ifthen}
\usepackage{geometry}
\usepackage{xcolor}
\usetikzlibrary{hobby,decorations.markings}
\usepackage{enumitem}   
\usepackage{float}
\usepackage{subcaption}
\usepackage{authblk}
\usepackage{babel}
\usepackage[T1]{fontenc}
\usepackage[utf8]{inputenc}
\usepackage{stmaryrd}

\newcommand{\Z}{\mathbb{Z}}

\newtheorem{theo}{Theorem}
\newtheorem{lemma}[theo]{Lemma}
\newtheorem{cor}[theo]{Corollary}

\def\arraypar#1{\parbox[c]{\textwidth - 2cm}{\centering #1}}

\RequirePackage{environ}
\NewEnviron{display}{
  \begin{equation}
    \begin{array}{c}
      \arraypar{\BODY}
    \end{array}
  \end{equation}
}

\pgfplotsset{compat=1.6}

\begin{document}

\title{Strict inequality for bond percolation on a dilute lattice with columnar disorder}

\author{M.R.~Hil\'ario$^*$ \hspace{2cm} M. S\'a$^{*,\dagger}$ \hspace{2cm} R.~Sanchis$^*$}
\date{}

\maketitle

\begin{abstract}
We consider a dilute lattice obtained from the usual $\mathbb{Z}^3$ lattice by removing independently each of its columns with probability $1-\rho$.
In the remaining dilute lattice independent Bernoulli bond percolation with parameter $p$ is performed.
Let $\rho \mapsto p_c(\rho)$ be the critical curve which divides the subcritical and supercritical phases.
We study the behavior of this curve near the disconnection threshold $\rho_c = p_c^{\text{site}}(\mathbb{Z}^2)$ and prove that, uniformly over $\rho$ it remains strictly below $1/2$ (the critical point for bond percolation on the square lattice $\mathbb{Z}^2)$.

\end{abstract}

\section{Introduction}

\subsection{Definition of the model and statement of the result}
\label{ss:intro}

The role played by disorder on the behavior of statistical mechanics models or interacting particle systems has been the subject of research in probability and mathematical physics for over 50 years.
In this context, an important question is to understand how dilution of the underlying lattice affects the critical point or the nature of the phase transition for ferromagnetic spin models \cite{griffiths68, georgii81, georgii84, aizenman87_2} or  percolation models \cite{menshikov87, aizenman91, chayes00}.
Also, since the pioneering work of McCoy and Wu \cite{mccoy68} on the two-dimensional Ising model with random impurities, much attention was drawn to the study of such models when defined on environments that present lower-dimensional disorder \cite{campanino91, kesten12, duminil-copin18} (see Section \ref{ss:motivation} for a discussion on this topic).

Motivated by this general type of questions, we study a percolation model on the cubic lattice $\mathbb{Z}^3$ that was introduced in \cite{grassberger17}.
It is defined in two steps: 
First vertical columns of sites of $\mathbb{Z}^3$ are selected independently with probability $1-\rho$ and all the sites lying in these columns are removed along with the edges incident to them.
In the sequel the bonds connecting neighboring sites in the remaining dilute lattice are declared open or closed independently with probability $p$ and $1-p$, respectively.

Before we state our main result, let us introduce the model mathematically.
Denote by $E(\mathbb{Z}^3)$ the set of edges of the $\mathbb{Z}^3$-lattice.
A bond percolation configuration is an element $\omega\in\{0, 1\}^{E(\mathbb{Z}^3)}$.
Fixed such a configuration, the edge $e$ is said to be \textit{open} if $\omega(e)=1$ and \textit{closed} if $\omega(e)=0$.
Any configuration $\omega$ can be identified to the subgraph of $\mathbb{Z}^3$ that comprises only the open edges.
Maximal connected components of this subgraph are called (open) clusters.

We regard the square lattice $\mathbb{Z}^2$ as a sublattice of $\mathbb{Z}^3$ that is, we identify $\mathbb{Z}^2$ to $\mathbb{Z}^2 \times \{0\}$.
An element $\Lambda\in\{0, 1\}^{\mathbb{Z}^2}$ is called an environment and we also identify it to the set $\{v\in \mathbb{Z}^2\colon \, \Lambda(v)=1\}$.
Given such an environment $\Lambda$ and a vertex $v \in \mathbb{Z}^2$, we say that  the column $\{v\}\times\mathbb{Z} \subset \mathbb{Z}^3$ is \textit{occupied}  if $\Lambda(v)=1$ and \textit{vacant} otherwise.
For each $\Lambda \in \{0,1\}^{\mathbb{Z}^2}$ let $\Lambda\times \mathbb{Z}$ be the subgraph of $\mathbb{Z}^3$ whose sites belong to occupied columns (with edges connecting any pair of such sites that are nearest neighbors).
Let $\nu_\rho$ be the probability measure on $\{0,1\}^{\mathbb{Z}^2}$ under which $\{\Lambda(x)\}_{x\in\mathbb{Z}^2}$ are independent Bernoulli random variables with parameter $\rho$ and let $\rho_c$ be the critical point for Bernoulli site percolation on $\mathbb{Z}^2$ so that, for $\rho \in (\rho_c,1)$ there exists a unique infinite connected component $C(\Lambda) \subset \Lambda$, $\nu_\rho$-a.s. (see \cite{grimmett99} for a comprehensive account on percolation).

Fixed $\Lambda \in \{0,1\}^{\mathbb{Z}^2}$, we denote by $\mathbb{P}^\Lambda_p$ the probability measure on $\{0, 1\}^{E(\mathbb{Z}^3)}$ under which $\{\omega(e)\}_{e\in E(\Lambda\times \mathbb{Z})}$ are independent Bernoulli random variables with mean $p$ and $\omega(e) =0$ for every edge $e$ belonging to $E(\mathbb{Z}^3) \setminus E(\Lambda\times \mathbb{Z})$.
Thus $\mathbb{P}^{\Lambda}_p$ is the law of bond percolation on the dilute $\mathbb{Z}^3$ lattice in the presence of quenched columnar disorder $\Lambda$.
The corresponding annealed measure on $\{0,1\}^{E(\mathbb{Z}^3)}$ is
\[
\mathbb{P}^\rho_p(\cdot):=\int_{\{0,1\}^{\mathbb Z^2}} \mathbb{P}^\Lambda_p (\cdot)d\nu_\rho(\Lambda).
\]
Notice that this model exhibits infinite-range dependencies under the annealed law.
In fact, the state of far away vertical bonds in the same column do not decorrelate with their distance.

Given $\Lambda$, let $p_c(\Lambda) = \inf \big\{p\in [0,1] \colon \mathbb{P}^{\Lambda}_p (\text{$o$ belongs to an infinite cluster}) >0 \big\}$ be the critical percolation threshold for the quenched law $\mathbb{P}^\Lambda_p$.
Similarly, we define $p_c(\rho)$ the critical threshold for the annealed law $\mathbb{P}_{p}^{\rho}$.
Standard ergodicity arguments can be employed to show that $p_c(\rho)=p_c(\Lambda)$, $\nu_\rho$-a.s.
Let us briefly list some immediate properties fulfilled by $\rho\mapsto p_c(\rho)$:
\begin{enumerate}
\item By standard coupling argument $p_c(\rho)$ is non-increasing.
For every $\rho\leq\rho_c$ we have $p_c(\rho)=1$ since $\nu_{\rho}$-almost all $\Lambda$ consists of a countable union of finite disjoint clusters so that $\Lambda\times\mathbb{Z}$ is a countable union of disjoint graphs extending infinitely in a single direction.
Furthermore, $p_c(1)$ equals the critical threshold for Bernoulli bond percolation on $\mathbb{Z}^3$.

\item For every $\rho>\rho_c$, $p_c(\rho)\leq \frac{1}{2}$.
In fact, $\nu_{\rho}$-almost all $\Lambda$ contains an infinite path.
The subset of $\Lambda\times\mathbb{Z}$ which projects orthogonally to such a path resembles a cramped infinite sheet.
It is indeed isomorphic to $\mathbb{Z_+}\times\mathbb{Z}$ whose critical point equals $1/2$, \cite{harris60, kesten80}. 
In particular, $\displaystyle{\lim_{\rho \to \rho_c +} p_c(\rho) \leq 1/2}$ and the curve is discontinuous at $\rho_c$.
\end{enumerate}

Figure \ref{f:spiral} shows the graph of $\rho \mapsto p_c(\rho)$ obtained numerically in \cite{grassberger17}.
As we decrease $\rho$ in the interval $(\rho_c,1]$ the (random) subgraph $C(\Lambda) \times \mathbb{Z}$ becomes thinner so that $p_c(\rho)$ increases.
The bound $\lim_{\rho \mapsto\rho_c +} p_c(\rho) \leq 1/2$ is not  sharp: As one approaches $\rho_c$ from above, the values of $p_c(\rho)$ fall strictly under $1/2$ uniformly in $\rho$.
A rigorous proof of this fact is the main contribution of the present article.
\begin{figure}[!htb]
\centering
\includegraphics[scale=0.4]{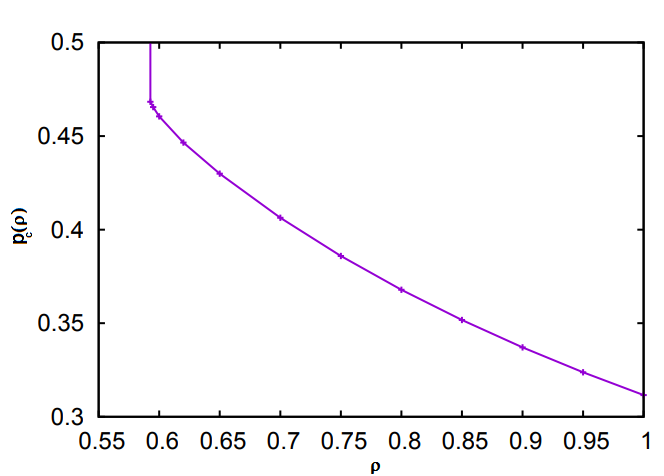}
\caption{This graph was reproduced from \cite{grassberger17} and shows critical curve $\rho\mapsto p_c(\rho)$.}
\label{f:critical_curve}
\end{figure}

\begin{theo}
\label{t:main}
There exists $\delta>0$ such that for every $\rho \in (\rho_c,1)$, $p_c(\rho)\leq\frac{1}{2}-\delta$.
\end{theo}

Phase transition for bond percolation on site-diluted graphs was studied by Chayes and Schonmann \cite{chayes00} under the name of mixed percolation.
They allow for much more general underlying graph $G$ (requiring that it be infinite, connected, with bounded degree and non-trivial critical points for Bernoulli percolation $\rho_c^{\text{site}}(G)<1$).
However, differently from us, they consider i.i.d.\ disorder, that is, the sites in the underlying graph $G$ are removed independently from one another with a given probability $1-\rho$.
Using differential inequalities techniques inspired by \cite{aizenman91} they show that the critical curve is strictly decreasing and Lipschitz-continuous on the interval $[\rho_c^{\text{site}}(G),1]$.
Moreover, it satisfies $p_c(\rho_c^{\text{site}}(G))=1$ implying that the curve is continuous in the whole interval $[0,1]$.

In our setting, Figure \ref{f:critical_curve} suggests strongly that, apart from the continuity at $\rho_c=\rho_c^{\text{site}}(\mathbb{Z}^2)$, the results obtained in \cite{chayes00} should also hold.
Note, however that the curve fails to be continuous at $\rho_c$ which is the disconnection threshold for the dilute lattice.
It is not very hard to show continuity on the other edge of the curve, that is at $\rho=1$. 
We present a sketch of the argument in Section \ref{s:conc_rema}.

The fact that we can only obtain results on the extremes of the interval $[\rho_c,1]$ is due to the lack of good tools to deal with strict monotonicity when the disorder presents long-range dependencies.
In particular, the method of Aizenman and Grimmett \cite{aizenman91} seems to be hard to apply in this context, since part of their arguments requires that one performs local modifications.

In our case, however, the application of such methods is hidden in the use of the so-called brochette percolation \cite{duminil-copin18} as we will explain with more details below.
In fact this result is based on a quantitative control of the prefactors appearing in the differential inequalities and require some knowledge of near-critical percolation in $2$-dimensions.
For this reason we are currently unable to prove similar results for higher dimensions.
For instance, if we remove $2$-d planes independently from the $\mathbb{Z}^4$, is it the case that the critical curve remains always below $p_c^{\text{bond}}(\mathbb{Z}^{3})$?

We close this section with a brief summary of the proof of Theorem \ref{t:main}.
It has two main inputs: A classical block argument and an enhancement-type argument which are presented in Sections \ref{s:block} and \ref{s:enha} respectively. 
Below we summarize these strategies.

Fix $\rho>\rho_c$.
The first input consists of constructing an infinite sequence of crossing events taking place inside an infinite set of overlapping boxes in $\mathbb{Z}^2$.
The side length of these boxes are chosen, depending on $\rho$, in such a way as to guarantee that they are all crossed simultaneously with positive probability.
On the event that these crossings occur, we carefully pick one of them in each box in such a way that the box is split into two regions, being one of them unexplored.

Moreover, one can concatenate the crossings in order to obtain an infinite occupied path in $\mathbb{Z}^2$.
The subgraph of $\mathbb{Z}^3$ that projects orthogonally to this infinite path is isomorphic to $\mathbb{Z}_+\times \mathbb{Z}$ so the critical point for Bernoulli bond percolation restricted to it equals $1/2$.
The idea now is to attach to this set a structure that will enhance the percolation process on it.
Perhaps the most interesting point in our argument is how to have this accomplished uniformly on $\rho>\rho_c$.
For that we use the so called brochette-percolation \cite{duminil-copin18} as we explain below.

The key idea is that, as one follows the infinite path, the environment on the left-hand side is always fresh.
This allows us to attach to the path an infinite sequence of evenly spaced threads of length one having one endpoint whose state is unexplored. 
Thus the states of the tip of these threads dominate an i.i.d.\ sequence of Bernoulli random variables with positive mean, uniformly bounded away from zero.

As their states are revealed, the path will become decorated with a sequence of randomly placed threads whose endpoints are occupied.
The edges that project to these threads will serve to enhance the percolation on the cramped sheet.
However, since their positions are random, the Aizenman and Grimmett enhancement-type arguments do not apply directly.
In order to show that the decorated cramped sheet is strictly better for percolation than a simple cramped sheet we use a mild modified of the so-called brochette percolation studied in \cite{duminil-copin18} (see  Section \ref{ss:2brochette}).

\subsection{Motivation and related models}
\label{ss:motivation}

As mentioned above, this work was partially motivated by the general question on how the introduction of impurities or defects along straight lines or affine hyperplanes  shifts the critical point or modifies important features of the phase transition in percolation models.

Since the work of McCoy and Wu \cite{mccoy68}, who studied the effect of the presence of random impurities on the phase transition for the two-dimensional Ising model, this type of question has been posed in a number of different contexts including percolation \cite{aizenman91, duminil-copin18, grassberger17, hoffman05, zhang94}, oriented percolation \cite{kesten12} and the Ising model \cite{campanino91, campanino91_2}.
Also similar questions appear for the contact process in random environments \cite{bramson91, madras94, newman96} since, when representing time as an extra dimension, the presence of spacial disorder manifests along a straight vertical line.

As mentioned before, the percolation model that we study in the present paper was introduced in \cite{grassberger17}.
Part of the interest in studying this model comes from the fact that, like in the so-called Bernoulli line percolation \cite{kantor86, hilario19, schrenk16, grassberger17_2}, it presents power-law decay of connectivity for some range of the parameters covering parts of the subcritical phase and the whole supercritical phase.

While power-law decay of connectivity is also exhibited in other percolation models that feature long-range correlations along the coordinate directions (e.g.\ Winkler's percolation \cite{winkler00} and corner percolation \cite{pete08}) the model we studied here and the Bernoulli line percolation model exhibit a transition from exponential to power-law decay occurs within the subcritical phase.
Perhaps this is one of the most interesting features of these models since it indicates that the transition from the disconnected to the connected phase may not be as sharp as the one presented in ordinary Bernoulli percolation \cite{aizenman87, menshikov86}.
It is worth to mention that such a behavior is not expected to be found in models where the dependencies decay sufficiently fast with the distance.

Other models of percolation on dilute lattice were considered before.
As discussed in Section \ref{ss:intro} in \cite{chayes00} the authors consider the case where site dilution is performed in an i.i.d.\ fashion.
In \cite{hoffman05} Hoffman studies the phase transition for a the square lattice where both vertical and horizontal edges are removed independently.
In this case, even to prove the existence of a non-trivial phase transition is a hard question. 
See also, \cite{hilario19_2} for a model where edges are removed in a more general way but only along one direction.

We finish this section commenting on a percolation model named brochette percolation studied in \cite{duminil-copin18}.
It is defined as follows: 
Starting from critical bond percolation on $\mathbb{Z}^2$ where no infinite cluster exist \cite{harris60}, vertical columns are selected independently with a certain density and the probability of opening the edges that lie on these columns is slightly increased while this probability is kept critical everywhere else.
Regardless of how small the enhancement is, it leads to the existence of an infinite cluster.
More than that, for any fixed enhancement strength, even if one discounts a suitable positive amount on the probability of opening the edges that do not belong to the enhanced columns,  the newly created infinite cluster still survives.
Not only it is related to the model we study here, since it presents infinite-range dependencies along vertical lines, but, as explained above, it will also be crucial for our arguments (see Section \ref{ss:2brochette}).

\subsection{Notation}
\label{ss:notation}

In this section we introduce some of the notation that will be used throughout the rest of the paper and revisit in a precise way some of the notation already introduced in Section \ref{ss:intro}.

For a finite set $A$, $|A|$ denotes its cardinality.
A subset $A\subset\mathbb{Z}$ is said $2$-\textit{spaced} if for every pair $i \neq j$ in $A$, one has $|i-j|\geq 2$.
We write $[n] :=\{1,\ldots, n\} \subset \mathbb{Z}$ and $[n]_0:=\{0,\ldots,n\}$.

Let $G=(V(G), E(G))$ be a graph, where $V(G)$ and $E(G)$ are, respectively, the \textit{set of vertices} and \textit{set of edges} of $G$.  
We will often abuse notation and write simply $G$ when referring to $V(G)$. 
We write $x\sim y$ if $\{x, y\}\in E(G)$. 
For $A\subset G$, we set $\partial A \vcentcolon=\{x\in G\setminus A: x\sim y \text{ for some $y\in A$}\}$.
For the sake of simplicity, we write $\partial{v} = \partial\{v\}$.

A \textit{path} $\gamma$ in $A\subset G$ is a sequence of vertices $(v_0,\ldots,v_n)$ with $v_i \sim v_{i+1}$, $v_i\in A$ and $v_i\not=v_j$ whenever $i\not=j$. 
Paths in $A$ will be often regarded as subsets of $G$.
The length of a path $\gamma$ is denoted $|\gamma|$ and its boundary is denoted by $\partial\gamma$.
The for the path $\gamma=(v_0,\ldots,v_n)$, the vertices $v_0$ and $v_n$ are called the extremes of $\gamma$.
We will denote $\Gamma$ the set of all the paths in $\mathbb{Z}^2$.
A subset of a path $\gamma \in \Gamma$ that is still a path is called a subpath of $\gamma$.
A path $\gamma \in \Gamma$ is said \textit{minimal} when its only subpath that has the same extremes as $\gamma$ is itself.

We will mainly consider $G=\mathbb{Z}^2$ or $G=\mathbb{Z}^3$, where for $d \in \{2,3\}$, $\mathbb{Z}^d = \big(V(\mathbb{Z}^d), E(\mathbb{Z}^d)\big)$ is the $d$-dimensional integer lattice.
It will be convenient to regard these lattices naturally embedded in $\mathbb{R}^d$. 
Recall that we also consider the $\mathbb{Z}^2$-lattice embedded in $\mathbb{Z}^3$ by identifying it to $\mathbb{Z}^2 \times \{0\}$. 
An edge $e \in E(\mathbb{Z}^3)$ is called a vertical edge if it projects to a single site in $\mathbb{Z}^2$ otherwise it is called a \textit{horizontal edge}.
We write $o$ in order to refer to the origin of $\mathbb{R}^d$.

We will be interested in studying random elements $\Xi \in \{0,1\}^{\mathbb{Z}_+}$, $\Lambda \in \{0,1\}^{\mathbb{Z}^2}$ and $\omega \in \{0,1\}^{E(\mathbb{Z}^3)}$ where the configuration spaces are equipped with the corresponding canonical $\sigma$-field generated by the cylinders set.
A site percolation on $\mathbb{Z}^2$ is a probability measure on $\{0,1\}^{\mathbb{Z}^2}$, for instance, $\nu_\rho$ as defined in Section \ref{ss:intro}.
A bond percolation on $\mathbb{Z}^3$ is a probability measure on $\{0,1\}^{E(\mathbb{Z}^3)}$, for instance, $\mathbb{P}^\Lambda$ and $\mathbb{P}^{\rho}_p$ as defined in Section \ref{ss:intro}.
For a fixed $\omega \in \{0,1\}^{E(\mathbb{Z}^3)}$ (resp.\ $\{0,1\}^{\mathbb{Z}^2}$, seen as a subgraph of $\mathbb{Z}^3$ (resp.\ a subset of $\mathbb{Z}^2$) we write $o \leftrightarrow \infty$ if the origin lies in an infinite connected component of $\omega$ (resp.\ of $\Lambda$).
For a fixed $\Lambda \in \{0,1\}^{\mathbb{Z}^2}$, called environment, we say that a path $\gamma=(v_0,\ldots,v_n)$ in $\mathbb{Z}^2$ is occupied if all of its vertices $v_i$ are occupied (that is, $\Lambda(v_i)=1$ for all $i$).

\section*{Acknowledgements}
MH was supported by CNPq grants “Projeto Universal” (307880/2017-6) and “Produtividade em Pesquisa” (406659/2016-8) and by FAPEMIG grant “Projeto Universal” (APQ-02971-17).
MS was supported by CAPES.
RS was supported by CNPq (grant 310392/2017-9), CAPES and FAPEMIG (PPM 0600/16).

\section{Block argument}
\label{s:block}

In this section we present the block argument leading to the construction of an infinite path decorated with open threads as mentioned in Section \ref{ss:intro}.
We start by defining the region of $\mathbb{Z}^2$ where we will attempt to find such a path.

Let $e_1 = (1,0) \in \mathbb{R}^2$  and $\theta_{\pi/2}$ be the rotation by $\pi/2$ about the origin in $\mathbb{R}^2$. 
For some integer $L = L(\rho) >1$, that will be chosen latter, consider the sequence of rectangles $R_i$ defined recursively as follows (see Figure \ref{f:spiral}):
\begin{enumerate}[label=(\roman*)]
\item  $w_1 := (0,0)$ and ${R}_1:=([0, \ 4L]\times[0, \ L])\cap \mathbb{Z}^2$,
\item $w_{n+1}:= w_n+2^{n+1}L\cdot\theta_{\pi/2}^{n-1}(e_1)$ and $ R_{n+1} := w_{n+1}+ 2^n \cdot \theta_{\pi/2}^n(R_1)$.
\end{enumerate}

\begin{figure}
\centering
\begin{tikzpicture}[scale=.5, every node/.style={scale=0.8}]
\filldraw[fill=black!20!, draw=black,  thick] (0,0)rectangle (3,0.75);
\filldraw[fill=black!20!, draw=black,  thick] (3,6)rectangle (-9,3);
\draw[pattern=north west lines, pattern color=gray] (1.5,0)rectangle (3,6);
\fill[pattern=north west lines, pattern color=gray] (-9,6)rectangle (-3,-1);
\draw(-9,6) to (-9, -1);
\draw(-3,6) to (-3, -1);
\node at (1,0.375) { $R_1$};
\node at (2.25,2.25) { ${{R_2}}$};
\node at (-2.3,4.5) { ${{R_3}}$};
\node at (-6, 2) { ${{R_4}}$};
\draw[ultra thick] (0,0) to (3,0);
\filldraw (0,0) circle (3pt);
\node at (0, -0.35) {\scriptsize $(0,0)$};
\node at (1.5,-0.5) { $b_1$};
\node at (3.5,3) { $b_2$};
\draw[ultra thick] (3,6) to (3,0);
\node at (-3,6.5) {$b_3$};
\draw[ultra thick] (3,6) to (-9,6);
\node at (-9.5,1.5) { $b_4$};
\draw[ultra thick] (-9,6) to (-9,-1);
\end{tikzpicture}
\caption{The sequence of overlapping rectangles $R_i$. Their bottom sides $b_i$ are represented by the thick segments whose concatenation form the broken line spiral.}
\label{f:spiral}
\end{figure}
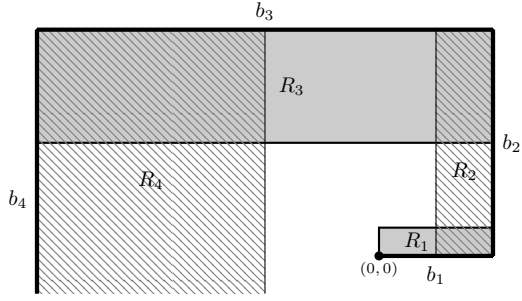
	
We also define $b_n=[w_n,w_{n+1}]$ where $[u,v]$ denotes the line segment between $u, v \in \mathbb{R}^2$, that is, $[u,v] = \{tv +(1-t)u\in \mathbb{R}^2;\, t \in [0,1]\}$. 
The concatenation of the segments $b_1, b_2,\ldots$ form the infinite broken-line spiral as shown in Figure \ref{f:spiral}.
The points $w_i$ are the ones at which this spiral breaks.
Notice that the rectangle $R_i$ overlap with $R_{i-1}$ and $R_{i+1}$.
The union of all the rectangles $R_i$ forms a spiral region inside which we will attempt to find an infinite minimal occupied path decorated with occupied threads.

It will be convenient to assign different reference frames to the successive rectangles $R_i$ so that $b_i$ is called the bottom side of $R_i$.
That is, the bottom side of each rectangle is defined to be its outermost side as seem from the center of the spiral.
Once the bottom side of $R_i$ is fixed, we naturally define its right, top and left sides following the order of appearance as we travel along its the boundary counterclockwise.
Note that for $i=1,5,9,\ldots$ the orientation coincides with the natural one so that the notion of right side $r_i$ coincides with the one in Section \ref{ss:notation}.
For $i=n \pmod{4}$ the orientation of the bottom side is rotated by $(n-1)\pi/2$ with respect to the original one.
With this convention, denote $r_i$ the right side of $R_i$.

We say that a path $\gamma$ crosses $R_i$ in the hard (resp.\ easy) direction if $\gamma$ is contained in $R_i$ and one of its extremes belong to the right (resp.\ top) side of $R_i$  and the other belongs to the left (resp.\ bottom) side of $R_i$.

Let us now assume $\rho>\rho_c$ and fix the value of $L = L(\rho)$ used in the definition of the rectangles $R_i$ above.
A straightforward adaptation of Lemma 4.12 in \cite{aiznman83} or Lemma 2.3 in \cite{chayes98} guarantees the existence of an integer $L=L(\rho)$ such that for every $i\in\mathbb{Z}_+$
\[
\nu_\rho\big(\text{$R_i$ is crossed in the hard direction}\big)\geq1-\dfrac{1}{2^{i+1}},
\] 
and therefore
\begin{equation}
\label{e:cross_hard}
\nu_\rho\Big(\bigcap_{i\in\mathbb{Z}_+}\big\{\text{$R_i$ is crossed in the hard direction}\big\}\Big)\geq \frac{1}{2}.
\end{equation}

\begin{figure}[htb]
\centering
\begin{tikzpicture}[scale=0.4] 
\filldraw[black!40!](0,0)--(0,8.8)--(5.4,8.8)--(5.4,7.2)--(4.8,7.2)--(4.8,4.8)--(3,4.8)--(3,7.8)--(1.2,7.8)--(1.2,3.6)--(1.8,3.6)--(1.8,1.2)--(0.6,1.2)--(0.6,0)--(0,0);
\filldraw[black!15!](1.2,0)--(1.2,0.6)--(1.8,0.6)--(2.4,1.2)--(2.4,2.4)--(5.4,2.4)--(5.4,1.8)--(7.8,1.8)--(7.8,3.6)--(12,3.6)--(12,2.4)--(10.2,2.4)--(10.2,1.8)--(9.6,1.8)--(9.6,0.6)--(10.8,0.6)--(10.8,1.2)--(14.4,1.2)--(14.4,2.4)--(16.8,2.4)--(16.8,1.8)--(18,1.8)--(18,0)--(1.2,0);
\draw[step = 6mm, help lines, densely dotted] (0, 0) grid (18,6);
\draw[step = 6mm, help lines, densely dotted] (0.1, 5.9) grid (5.9,8.8);
\draw[thick] (0,0) rectangle (18,6);
\draw[ultra thick] (0,0) -- (18,0);
\draw[ultra thick] (18,0) -- (18,6);
\draw[dashed, thick] (0,0) -- (0,8.8);
\draw[dashed, thick] (6,0) -- (6,8.8);
\draw (5.4,8.8)--(5.4,7.2)--(4.8,7.2)--(4.8,4.8)--(3,4.8)--(3,7.8)--(1.2,7.8)--(1.2,3.6)--(1.8,3.6)--(1.8,1.2)--(0.6,1.2)--(0.6,0);
\draw[ultra thick] (2.4,2.4)--(5.4,2.4)--(5.4,1.8)--(7.8,1.8)--(7.8,3.6)--(12,3.6)--(12,2.4)--(10.2,2.4)--(10.2,1.8)--(9.6,1.8)--(9.6,0.6)--(10.8,0.6)--(10.8,1.2)--(14.4,1.2)--(14.4,2.4)--(16.8,2.4)--(16.8,1.8)--(18,1.8);
\foreach \x/\y in {5.4/6,5.4/5.4,5.4/4.8,4.8/4.2,4.2/4.2,3.6/4.2,3/4.2,2.4/4.8,2.4/5.4,2.4/6,1.8/6,1.8/5.4,1.8/4.8,1.8/4.2,2.4/3.6,2.4/3,2.4/2.4,2.4/1.8,2.4/1.2,1.8/0.6,1.2/0.6,1.2/0} {
	\filldraw[white] (\x, \y) circle(3pt);
	\draw[black] (\x, \y) circle(3pt);}
\node[below] at (9,0) {$b_i$};
\node at (1.8,7) {$\gamma$};
\node at (12.6,3) {$\gamma'$};
\node at (5.4,4) {$l_\gamma$};
\node at (19,0) {$R_i$};
\node at (7.3,7.5) {$R_{i-1}$};
\node at (18.5,3) {$r_i$};
\node at (15.4,0.6) {\footnotesize$\mathcal{D}(R_i,\gamma,\gamma')$};
\end{tikzpicture}
\caption{The bottom and right sides of $R_i$ are represented as thick segments labeled $b_i$ and $r_i$, respectively.
The path $\gamma$ has a subpath that crosses $R_i$ in the easy way.
The set $l_\gamma$, whose vertices are represented by $\circ$, belong to does not intersect the right $r_i$. 
The path $\gamma'$ crosses $R_i$ from $l_\gamma$ to $r_i$. 
The light gray colored region corresponds to $\mathcal{D}(R_i,\gamma,\gamma').$}\label{f:lowest_crossing}
\end{figure}
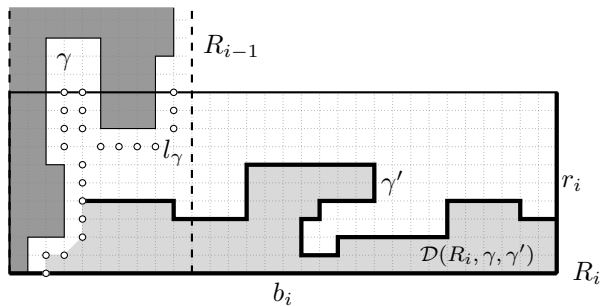

Suppose that a path $\gamma$ in $\mathbb{Z}^2$ contains a subpath that crosses  $R_i$ in the easy direction and that  $\partial\gamma\cap r_i=\emptyset$. 
We set $l_\gamma:=\partial \gamma\cap R_i$ which does not need to be connected in the sense of $\mathbb{Z}^2$ but nevertheless, it spans $\mathbb{R}_i$ in the easy direction (see Figure \ref{f:lowest_crossing} for an illustration of $l_\gamma$).
Let $\Gamma(R_i, \gamma)$ be the set of all minimal paths $\gamma'$ contained in $R_i$, such that $|\gamma'\cap l_\gamma|=1$ and that has one extreme in $l_{\gamma}$ and the other in $r_i$.
Define $\{l_\gamma \overset{R_i}{\longleftrightarrow}r_i\}$ the event that $l_{\gamma}$ and $r_i$ are \textit{connected} in $R_i$, i.e.\ the event that $\Gamma(R_i, \gamma)$ contains at least one occupied path.
Given $\Lambda\in \{l_\gamma \overset{R_i}{\longleftrightarrow}r_i\}$, denote by $\mathcal{L}_{R_i, \gamma}$ the lowest occupied minimal path in $\Gamma(R_i, \gamma)$.
For $\gamma'\in\Gamma(R_i, \gamma)$, denote by $\mathcal{D}(R_i, \gamma, \gamma')$  the region of $R_i$ located below $\gamma'$ and to the right of $l_{\gamma}$ including $\gamma'$ (see Figure \ref{f:lowest_crossing}). 
Note that, for all $\gamma'\in\Gamma(R_i, \gamma)$, the event $\{ \mathcal{L}_{R_i, \gamma}=\gamma'\}$ is measurable with respect to the state of the vertices in  $\mathcal{D}(R_i, \gamma, \gamma')$.

We are ready to construct the relevant crossing events to be used in our block argument.
Let $\gamma_0=\{-1\}\times[0,L]$ and assume that $\{l_{\gamma_0} \overset{R_1}{\longleftrightarrow}r_1\}$ occurs (note that $l_{\gamma_0} = \{0\}\times [0,L]$, the left side of $R_1$).
We can extract $\mathcal{L}_{R_1, \gamma_0}$ the lowest minimal occupied path connecting $l_{\gamma_0}$ to $r_1$.
Now, once we condition on $\mathcal{L}_{ R_1, \gamma_0}=\gamma_1$ and since such a path $\gamma_1$ necessarily has a subpath that crosses $R_2$ in the easy direction without intersecting $r_2$, we can consider the event $\{l_{\gamma_1} \overset{R_2}{\longleftrightarrow}r_2\}$.
On this event we can extract the lowest minimal occupied path $\mathcal{L}_{ R_2, \gamma_1}$.
We continue progressively:
Conditional on $\{\mathcal{L}_{ R_i, \gamma_{i-1}} =\gamma_i\}$ and $\{l_{\gamma_i} \overset{R_{i+1}}{\longleftrightarrow}r_{i+1}\}$ we denote $\mathcal{L}_{ R_{i+1}, \gamma_i}$ the lowest minimal occupied path inside $R_{i+1}$ crossing from $l_{\gamma_i}$ to $r_{i+1}$.

We say that a sequence of paths $\{\gamma_i\}_{i \in [k]}$ in $\Gamma$ is allowed if $\gamma_i\in \Gamma(R_i, \gamma_{i-1})$ for all $i\in[k]$.
For each $k$ we set
\begin{equation}
\label{e:crossing_events}
\Delta_k:=\bigcup_{\{\gamma_i\}_{i\in[k]} \text{ allowed}}\Big\{ \text{$\mathcal{L}_{R_i, \gamma_{i-1}}=\gamma_i$ for every $i \in [k]$}\Big\}
\end{equation}
where the union is disjoint.
Since the $k$ first elements in an allowed sequence $\{\gamma_i\}_{i\in[k+1]}$ still form an allowed sequence, we have $\Delta_{k+1} \subset \Delta_k$.

Moreover,
\[
\bigcap_{1\leq i\leq k}\big\{ R_i \text{ is crossed in the hard direction}\big\}\subset \Delta_k,
\]
which, in view of \eqref{e:cross_hard} implies
\begin{equation}
\label{e:nu_rho_delta_k}
\nu_\rho\Big(\bigcap_{k\in\mathbb{Z}_+}\Delta_k\Big) \geq \frac{1}{2}.
\end{equation}

The procedure is summarized in the following lemma. 
Recall that $\Gamma$ denotes the set of all paths in $\mathbb{Z}^2$.
Let $\mathcal{F}_k = \sigma(\Lambda(v)\colon v\in \cup_{j=1}^k R_j)$ be the $\sigma$-field generated by the random element $\Lambda$ restricted to the first $k$ rectangles so that $(\mathcal{F}_k)_{k\in\mathbb{Z}_+}$ defines a filtration.
Notice that the  sequence $\{\Delta_k\}_{k \in\mathbb{Z}_+}$ is adapted, that is $\Delta_k \in \mathcal{F}_k$ for every $k$.

\begin{lemma}
\label{l:main} 
For every $\rho>\rho_c$ there exists an integer $L=L(\rho)>0$ such that, the corresponding adapted sequence of crossing events $\{\Delta_k\}_{k\in\mathbb{Z}_+}$ defined in \eqref{e:crossing_events} satisfies
\begin{enumerate}[label=(\roman*)]
\item\label{ml1}  $\nu_\rho( \cap_k \Delta_k)\geq\frac{1}{2}$.
\item\label{ml2} There for every $k$, there exists a function $\Phi_k:\Delta_k\rightarrow \Gamma$, such that, for every $\Lambda \in \Delta_k$
  $\Phi_k(\Lambda)$ is an occupied minimal path having one extreme in $\partial\gamma_0$ and the other in $r_k$ (the right side of the $k$-th rectangle in the spiral, $R_k$).

\item \label{ml3} For every path $\gamma=(v_0,\ldots,v_n) \in \Phi_k(\Delta_k)$,  there is a $2$-spaced set $\mathcal{B}_\gamma\subset [n]_0$ and a set $\mathcal{H}_\gamma\subset\partial \gamma$ such that
 for every $i\in [n]_0\setminus \mathcal{B}_\gamma$, $\partial v_i\cap \mathcal{H}_\gamma\not=\emptyset$.
 Moreover, $\sigma(\Lambda(v) \colon v\in \mathcal{H}_\gamma)$, is independent of $\{\Phi_k = \gamma\}$.

\end{enumerate}
\end{lemma}

Before we prove Lemma \ref{l:main} let us clarify its statement.
On the event $\Delta_k$, the function $\Phi_k$ selects a minimal occupied path $\gamma$ with one extreme in $\partial\gamma_0$ and the other in $r_k$ by concatenating the crossings in successive rectangles.
For each vertex $y \in \mathcal{H}_\gamma$, there exists $v_i \in \gamma$ such that  $y \in \partial v_i \setminus \cup_{i=1}^k\mathcal{D}(R_i,\gamma_{i-1},\gamma_i)$.
We think of $y$ as being an endpoint of a thread of length one that will be attached to $v_i$.
Vertices of $\gamma$ which are not attached to any thread are indexed by the set $\mathcal{B}_{\gamma}$.
The fact that $\mathcal{B}_{\gamma}$ is $2$-spaced implies that, for each pair of consecutive vertices in $\gamma$, at least one of these vertices has a thread attached.
Finally, the $\nu_\rho$-state of the vertices $y \in \mathcal{H}_{\gamma}$, is independent of the event that $\gamma$ was the path selected.

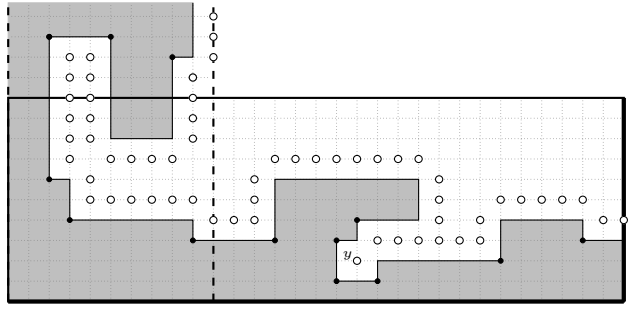
\begin{figure}
	\centering
	\begin{tikzpicture}[scale=0.45]	\filldraw[black!25!](0,0)--(0,8.8)--(5.4,8.8)--(5.4,7.2)--(4.8,7.2)--(4.8,4.8)--(3,4.8)--(3,7.8)--(1.2,7.8)--(1.2,3.6)--(1.8,3.6)--(1.8,2.4)--(3,2.4)--(5.4,2.4)--(5.4,1.8)--(7.8,1.8)--(7.8,3.6)--(12,3.6)--(12,2.4)--(10.2,2.4)--(10.2,1.8)--(9.6,1.8)--(9.6,0.6)--(10.8,0.6)--(10.8,1.2)--(14.4,1.2)--(14.4,2.4)--(16.8,2.4)--(16.8,1.8)--(18,1.8)--(18,0)--(0,0);
	\draw[step = 6mm, help lines, densely dotted] (0, 0) grid (18,6);
	\draw[step = 6mm, help lines, densely dotted] (0.1, 5.9) grid (5.9,8.8);
	\draw[thick] (0,0) rectangle (18,6);
	\draw[ultra thick] (0,0) -- (18,0);
	\draw[ultra thick] (18,0) -- (18,6);
	\draw[dashed, thick] (0,0) -- (0,8.8);
	\draw[dashed, thick] (6,0) -- (6,8.8);
	\draw (5.4,8.8)--(5.4,7.2)--(4.8,7.2)--(4.8,4.8)--(3,4.8)--(3,7.8)--(1.2,7.8)--(1.2,3.6)--(1.8,3.6)--(1.8,2.4)--(3,2.4)--(5.4,2.4)--(5.4,1.8)--(7.8,1.8)--(7.8,3.6)--(12,3.6)--(12,2.4)--(10.2,2.4)--(10.2,1.8)--(9.6,1.8)--(9.6,0.6)--(10.8,0.6)--(10.8,1.2)--(14.4,1.2)--(14.4,2.4)--(16.8,2.4)--(16.8,1.8)--(18,1.8);
	\foreach \x/\y in {6/8.4,6/7.8,6/7.2,5.4/6.6,5.4/6.6,5.4/6,5.4/5.4,5.4/4.8,4.8/4.2,4.2/4.2,3.6/4.2,3/4.2,2.4/4.8,2.4/5.4,2.4/6,2.4/6.6,2.4/7.2,1.8/7.2,1.8/6.6,1.8/6,1.8/5.4,1.8/4.8,1.8/4.2,2.4/3.6,2.4/3,3/3,3.6/3,4.2/3,4.8/3,5.4/3,6/2.4,6.6/2.4,7.2/2.4,7.2/3,7.2/3.6,7.8/4.2,8.4/4.2,9/4.2,9.6/4.2,10.2/4.2,10.8/4.2,11.4/4.2,12/4.2,12.6/3.6,12.6/3,12.6/2.4,10.8/1.8,11.4/1.8,12/1.8,12.6/1.8,13.2/1.8,13.8/1.8,10.2/1.2,13.8/2.4,14.4/3,15/3,15.6/3,16.2/3,16.8/3,17.4/2.4,18/2.4} {
	\filldraw[white] (\x, \y) circle(3pt);
	\draw[black] (\x, \y) circle(3pt);}
\node[left] at (10.35,1.35) {\tiny$y$};
\draw[fill] (5.4,1.8) circle(2pt);
\draw[fill] (1.8,2.4) circle(2pt);
\draw[fill] (1.2,3.6) circle(2pt);
\draw[fill] (1.2,7.8) circle(2pt);
\draw[fill] (1.2,7.8) circle(2pt);
\draw[fill] (3.0,7.8) circle(2pt);
\draw[fill] (4.8,7.2) circle(2pt);
\draw[fill] (7.8,1.8) circle(2pt);
\draw[fill] (10.2,2.4) circle(2pt);
\draw[fill] (9.6,1.8) circle(2pt);
\draw[fill] (9.6,0.6) circle(2pt);
\draw[fill] (10.8,0.6) circle(2pt);
\draw[fill] (14.4,1.2) circle(2pt);
\draw[fill] (16.8,1.8) circle(2pt);
	\end{tikzpicture}
	\caption{We illustrate part of the path $\gamma = \Phi_k$.
	The grey shaded area represents $D_k$.
	The black-filled dots are the sites $v_i$ for $i \in \mathcal{B}_\gamma$.
	The sites in $\mathcal{H}_\gamma$ are represented as $\circ$ and their states are independent of the event $\{\Phi_k =\gamma$\}. They can neighbor up to $4$ sites in $\gamma$, as for example, the site $y$ illustrated in the picture.}
	\label{f:lowest_crossing_2}
\end{figure}

\begin{proof}[Proof of Lemma \ref{l:main}]

Item $(i)$ is just \eqref{e:nu_rho_delta_k}.

Fix $\Lambda\in\Delta_k$, and let $\{\gamma_i\}_{i\in[k]}$ be the unique allowed sequence such that $\mathcal{L}_{R_i,\gamma_{i-1}} = \gamma_i$ for every $i \in [k]$ and let  $u_i:=\gamma_i\cap \partial\gamma_{i-1}$ be the extreme $\gamma_i$ that also lie in $l_{\gamma_{i-1}}$.

Now we define $\Phi_k(\Lambda)$ as being the path that starts at $u_1$, goes along $\gamma_1$ until hitting a neighbor of $u_2$, then follows along $\gamma_2$ until hitting a neighbor of the $u_3$ and so on until the last step when it starts at $u_{k-1}$ and follows along $\gamma_{k}$ until hitting $r_k$.
It is clear that $\Phi_k(\Lambda)$ is an occupied minimal path contained in $\cup_{i=1}^k R_k$ and that it has one extreme on $\partial\gamma_0$  and the other on $r_k$.

\begin{figure}
\hspace*{3.2cm}
\vspace*{5mm}
\begin{tikzpicture}[scale=0.5, every node/.style={scale=0.7}]
	\path [ultra thick, pattern=north west lines, pattern color=black!30!](3.2,8) to [out=-110,in=90](3,7) to [out=-150,in=-20, rounded corners] (-3,6) to [out=110,in=-120, rounded corners] (-1.7,6.8) to [out=165, in=10, rounded corners] (-9,7) to [out=-110,in=90] (-8,5.7) to [out=-110,in=30] (-12,6)--(-12,8)--(4,8)--(3.2,8);
	\path [ultra thick, pattern=north west lines, pattern color=black!30!] (0,0) -- (0,0.3) to [out=-30,in=170, rounded corners] (2, 0.6) to [out=-10, in=180, rounded corners] (3, 0.4) to [out=0, in=152](4,0.2)--(4,0)--(0,0);
	\path [ultra thick, pattern=north west lines, pattern color=black!30!] (3,0.4) to [out=0,in=152] (4, 0.2)--(4,8)--(3.2,8)  to [out=-110,in=90] (3,7) to [out=-90,in=85, rounded corners] (3.5,3.4) to [out=-175,in=20, rounded corners] (2.7,4.5)to [out=-105,in=80](3,0.4);
	\draw (0,0)rectangle (4,1);
	\draw (2,0)rectangle (4,8);
	\draw (4,4)rectangle (-12,8);
	\draw[ultra thick, black] (0,0.4) to [out=-30,in=170, rounded corners] (2, 0.6) to [out=-10, in=180, rounded corners] (3, 0.4);
	\draw[thick](3, 0.4) to [out=0, in=210](4,0.5);
	\draw[thick](3.2,8)  to [out=-110,in=90] (3,7);
	\draw[ultra thick, black] (3,7) to [out=-90,in=85, rounded corners] (3.5,3.4) to [out=-175,in=20, rounded corners] (2.7,4.5)to [out=-105,in=80](3,0.4);
	\draw[ultra thick, black] (3,7) to [out=-150,in=-20, rounded corners] (-3,6) to [out=110,in=-120, rounded corners] (-1.7,6.8) to [out=165, in=10, rounded corners] (-9,7) to [out=-110,in=90] (-8,5.7) to [out=-110,in=30] (-12,6);
	\node at (0,7) {$D_k$};
	\node at (-3.3,5.8) {$\Phi_k$};
	\end{tikzpicture}
	\caption{We illustrate the three first rectangles in the spiral region with the lowest minimal crossings within them. 
	We highlight the path $\Phi_3$ formed by the concatenation of parts of these crossings. 
	The shaded region is $D_3$.}
	\label{f:spiral_path}
	\end{figure}
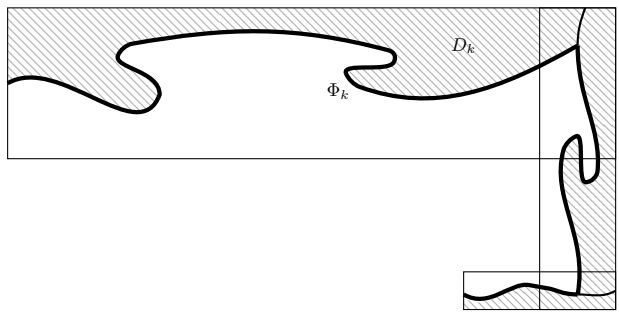

Also, $\Phi_k(\Lambda)$ divides $\cup_{i=1}^k R_i$ into two regions being  
\[
D_k(\Lambda)\subset \cup_{i=1}^k\mathcal{D}(R_i,\gamma_i, \gamma_{i-1}) 
\]
the one located between $\Phi_k(\Lambda)$ and the outermost part of the spiral.
Moreover, $\{\Phi_k = \gamma\}$ is measurable with respect to $\sigma(\Lambda(v) \colon v \in D_k)$.

Now for some $\Lambda\in \Delta_k$  let $\Phi_k(\Lambda)=\gamma=(v_0,\ldots,v_n)$ and define
\begin{equation}
\label{e:B_gamma}
B_{\gamma}:= \{i \in [n]_0 \colon \partial{v}_i \subset D_k(\Lambda)\}.
\end{equation}
We claim that $B_\gamma$ is a $2$-spaced set. 
Indeed, if $\partial v_i$ and $\partial v_{i+1}$ were both contained in $D_k(\Lambda)$, then we would have $v_{i-1}\sim v_{i+2}$ contradicting the minimality of $\gamma$.

Choose any ordering of $\mathbb{Z}^2$.
For every $i\in [n]\setminus B_{\gamma}$ let $y_i$ be the earliest vertex in $\partial v_i \setminus D_k(\Lambda)$ and define
\begin{equation}\label{H}
\mathcal{H}_\gamma:=\bigcup_{i\in [n]_0\setminus \mathcal{B}_\gamma}\{y_i\} \subset \partial \gamma \cap \big[\cup_{i=1}^k R_i\setminus D_k(\Lambda)\big].
\end{equation}

Notice that $\sigma(\Lambda(v) \colon v\notin D_k(\Lambda))$  is independent of $\{\Phi_k = \gamma\}$ since the latter is measurable with respect to the states of vertices in $D_k(\Lambda)$.
\end{proof}

We stress once more that, for the rest of the text we will fix $L = L(\rho)$ as given by Lemma \ref{l:main}. 
We will also abuse notation and write $\Delta_\gamma=\Phi_k^{-1}(\gamma)$, omitting the dependency on $k\in\Z_+$. 
It is the case that $\{\Delta_\gamma\}_{\gamma\in\Phi_k(\Delta_k)}$ form a partition of $\Delta_k$.

\section{Coupling with brochette percolation}
\label{s:enha}

In the previous section we showed that, with positive probability, under $\nu_\rho$, we can find an infinite occupied path $\gamma$ that is decorated with a set of threads $\mathcal{H}_\gamma$ whose states are unexplored.
Roughly speaking, in order to prove Theorem \ref{t:main} it remains to show that, uniformly on the realization of such pair $(\gamma, \mathcal{H}_\gamma)$, Bernoulli bond percolation on $(\gamma \cup \mathcal{H}_\gamma)\times \mathbb{Z} \subset \mathbb{Z}^3$ has a lower critical point than on $\mathbb{Z}_+\times \mathbb{Z}$.
While this results sound intuitively clear, it is not straightforward due to the random location of the occupied threads that should be used to enhance the percolation process.
As already mentioned above, the comparison will be made by mapping to a bond percolation model on $\mathbb{Z}^2$ that resembles the so-called brochette percolation model \cite{duminil-copin18}.
In Section \ref{ss:2brochette} we define the brochette percolation model precisely.
Then, in Section \ref{ss:cscheme} we will construct the desired coupling and prove Theorem \ref{t:main}.

\subsection{Brochette percolation}
\label{ss:2brochette}
 
We start this section presenting the so-called brochette percolation.
Given $\Xi\in \{0, 1\}^{\mathbb{Z}_+}$, for $i\in \mathbb{Z}_+$, the column $\{i\}\times \mathbb{Z}$ is said \textit{strong} if $\Xi(i)=1$ and $\textit{weak}$ otherwise.
Fix two parameters $0<p\leq q<1$.
Edges of the lattice $\mathbb{Z}_+\times \mathbb{Z}$ will be declared open or closed independently as follows.
Vertical edges $e\in E(\Z_+\times\Z)$ whose endpoints lie in strong columns, are declared open with probability $q$ and closed with probability $1-q$.
Every other edge in $E(\Z_+\times\Z)$ is declared open with probability $p$ and closed with probability $1-p$.
This gives rise to an inhomogeneous bond percolation on $\mathbb{Z}_+\times \mathbb{Z}$ whose law we denote $P^{\Xi}_{p,q}$.

For a measure $\mu$ on $\{0,1\}^{\mathbb{Z}_+}$ let us define
\begin{equation}\label{brochette}
{P}^{\mu}_{p,q}(\cdot):=\int {P}^{\Xi}_{p, q} (\cdot)d{\mu}(\Xi).
\end{equation}

In \cite{duminil-copin18}, $\{\Xi(i)\}_{i\in\mathbb{Z}_+}$ are assumed to be independent Bernoulli random variables with mean $u>0$ that is, $\mu = \mu_u$, where
\begin{equation}
\label{e:mu_u}
\mu_u :=\otimes_{i\in\mathbb{Z}_+} [(1-u)\delta_0 + u \delta_1].
\end{equation}
The resulting measure $P^{\mu_u}_{p,q}$ is called the brochette percolation on $\mathbb{Z}_+ \times \mathbb{Z}$.
The parameter $u$ should be understood as a density of enhanced lines.
The main result in \cite{duminil-copin18} is:
\begin{theo}[\cite{duminil-copin18}, Theorem 1, p.~481]
\label{t:brochette}
For every $u\in(0,1)$ and $\varepsilon\in (0,1/2)$ there exists $\delta = \delta(u, \varepsilon)>0$ such that 
\[
P^{\Xi}_{1/2-\delta,1/2+\varepsilon}\big(o\leftrightarrow\infty\big) >0
\]
for $\mu_u$-almost all $\Xi$.
\end{theo}
Strictly speaking, in \cite{duminil-copin18} the model was defined on the $\mathbb{Z}^2$-lattice. 
However the exact same proof presented there works for the half-space $\mathbb{Z}_+ \times \mathbb{Z}$.

An immediate consequence of Theorem \ref{t:brochette} is: For every $u\in(0,1)$ and every $\varepsilon \in (0,1/2)$, there exists $\delta = \delta(u,\varepsilon)>0$ and $\alpha = \alpha(u, \varepsilon)>0$ such that for every $p \geq 1/2-\delta$,  $P^{\mu_u}_{p,1/2+\varepsilon}(o\leftrightarrow\infty) > \alpha$.

For our purposes it will be useful to force $\Xi(i)$ to vanish at some indices $i$.
For this reason, let us consider for every $B \subset\mathbb{Z}_+$,  the measure 
\begin{equation}
\label{e:mu_ub}
\mu_{u,B}:=\otimes_{i\in\mathbb{Z}_+ \setminus B} \big[(1-u)\delta_0 + u\delta_1\big] \times \otimes_{i\in B} \delta_0.
\end{equation}

The following theorem is an adaptation of Theorem \ref{t:brochette}, the only difference being that instead of deciding the position of the strong columns independently we only allow for strong columns outside a $2$-spaced set.
This amounts to replace $\mu_u$ in \eqref{e:mu_u} by $\mu_{u,B}$ in \eqref{e:mu_ub}.

\begin{theo} 
\label{t:brochette_new} 
For every $\varepsilon \in (0,1/2)$ and $u\in(0,1)$, there exist $\delta=\delta(\varepsilon, u) >0$ and $\alpha=\alpha(\varepsilon, u)>0$, such that for every $2$-spaced set $B\subset\mathbb{Z}_+$ and every $p\geq 1/2-\delta$ 
\[
P^{\mu_{u,B}}_{p, \frac{1}{2}+\varepsilon}\big(o\leftrightarrow \infty\big)>\alpha.
\]
\end{theo}
We refrain from writing down a proof here since it would follow exactly the same lines as Theorem \ref{t:brochette} with very straightforward modifications.

\begin{cor}
\label{cor:brochette_new} 
For every $\varepsilon \in (0,1/2)$, $u\in(0,1)$ and $n \in \mathbb{Z}_+$, there exist $\delta=\delta(\varepsilon, u) >0$ and $\alpha=\alpha(\varepsilon, u)>0$, such that for every $2$-spaced set $B\subset[n]_0$ and every $p\geq 1/2-\delta$
\begin{equation}
\label{e:brochette_new_n}
P^{\mu_{u,B}}_{p, \frac{1}{2}+\varepsilon}\big(\text{$o\leftrightarrow \{n\}\times \mathbb{Z}$ in $[n]_0\times \mathbb{Z}$}\big)>\alpha.
\end{equation}
\end{cor}

\subsection{Proof of Theorem \ref{t:main}}
\label{ss:cscheme}

In this section we compare $\mathbb{P}^{\rho}_{p}$ restricted to a certain random subset of $\mathbb{Z}^3$ and the brochette percolation  $P^{\mu_{u,B}}_{p,q}$ on $\mathbb{Z}_+\times\mathbb{Z}$ with an appropriate choice of parameters $u$ and $q$ and of a random 2-spaced subset $B \subset \mathbb{Z}_+$.
Then we use this comparison to prove Theorem \ref{t:main}.

Fix $\rho>\rho_c$ and $p\in(0,1)$.
Let $L = L(\rho)$ be given as in Lemma \ref{l:main}.
Define $u=u(\rho)$ and $q=q(p)$ as
\begin{equation}
\label{e:u_q_p}
u:=1-(1-\rho)^{1/4}, \,\,\,\,\,\,\,\,\, q:=p+\Big[(1-p)\big(1-(1-p)^{1/2})\big)^2\big(1-(1-p)^{1/4})\big)\Big].
\end{equation}
For $k\in\mathbb{Z}_+$ and a path $\gamma\in \Phi_k(\Delta_k)$, recall the definitions of $\mathcal{B}_\gamma$ as in Lemma \ref{l:main} and  $\Delta_\gamma=\Phi^{-1}_k(\gamma)$.
\begin{lemma}
\label{l:coupling}
Let $\rho > \rho_c$ and $p\in (0,1)$.
Fix $\gamma=(v_0,\ldots,v_n) \in \Phi_k(\Delta_k)$ and let $u=u(\rho)$ and $q=q(p)$ be given as in \eqref{e:u_q_p}. 
Then,
\begin{eqnarray}
\mathbb{P}^\rho_p\big(\text{$(v_0,0) \leftrightarrow \{v_n\}\times\mathbb{Z}$ in $(\gamma\cup\mathcal{H}_\gamma)\times\Z$} \mid \Delta_\gamma \big)\geq {P}^{\mu_{u,\mathcal{B}_\gamma}}_{p,q}\big(\text{$o\leftrightarrow \{n\}\times\mathbb{Z}$ in $[n]_0\times \mathbb{Z}$}\big).
\label{e:coupling}
\end{eqnarray}
\end{lemma}

The proof of the above lemma relies on a simple coupling of a percolation process $\mathbb{P}^\rho_p$ restricted to $(\gamma\cup\mathcal{H}_\gamma)\times \mathbb{Z}$ and the brochette percolation $P^{\mu_{u,\mathcal{B}_\gamma}}_{p,q}$ restricted to $[n]_0\times \mathbb{Z}$ as we sketch now.

Fixed $\gamma = (v_0,\ldots, v_n)$, we wish to associate each site $v_i$ with $i\in [n]_0\setminus\mathcal{B}_\gamma$ to an adjacent thread $y\in \mathcal{H}_\gamma$.
Since each thread $y$ may neighbor up to $4$ sites in $\gamma$ (see Figure \ref{f:lowest_crossing_2}), we split them into $4$ quarter threads and color each one of them blue independently with probability $u$.
Therefore, we can now assign to each $v_i$ with $i\in [n]_0\setminus \mathcal{B}_\gamma$ a neighboring quarter thread in an injective way.
We now couple the environments $\Lambda$ and $\Xi$ in the two processes.
Since we want to condition on $\Delta_\gamma$, we fix $\Lambda(v)=1$ for every $v \in\gamma$.
Then declare each $y \in \mathcal{H}_\gamma$ occupied (that is, $\Lambda(y) =1$) if at least one of the $4$ quarter threads resulting from it is colored blue and unoccupied (that is, $\Lambda(y)=0$) otherwise.
In order to construct $\Xi \in \{0,1\}^{\mathbb{Z}_+}$, we force $\Xi(j)=0$ if $j$ belongs to $\mathcal{B}_\gamma$ and, for each $j$ outside $\mathcal{B}_\gamma$ we let $\Xi(j)=1$ if the quarter thread assigned to the site $v_j$ is colored blue.
One can check that $\{\Lambda(v); v\in \gamma\cup\mathcal{H}_\gamma\}$ is distributed as $\nu_\rho(\, \cdot \mid \Delta_\gamma)$ and that $\{\Xi(j)\}_{j\in [n]_0}$ is distributed as $\mu_{u,\mathcal{B}_\gamma}$.

Conditional on the above realizations of $\Lambda$ and $\Xi$ we now construct the desired bond percolation processes on $(\gamma \cup \mathcal{H}_\gamma)\times \mathbb{Z}$ and $[n]_0 \times \mathbb{Z}$.
We start by defining the states of the edges in $E(\gamma \cup \mathcal{H}_\gamma)\times \mathbb{Z}$ independently as follows.
All the edges in $E(\gamma \times \mathbb{Z})$ are declared open (resp.\ closed) with probability $p$ (resp.\ $1-p$).
Now let us construct the state of edges that have at least one endpoint in $\mathcal{H}_\gamma \times \mathbb{Z}$.
If this endpoint projects orthogonally into an unoccupied thread $y\in \mathcal{H}_\gamma$, then the edge is declared closed whereas, if it projects to an occupied thread, then we decide its state according to the following procedure: 
\begin{enumerate}
    \item If $f=\{(y,z),(y,z+1)\}$ is a vertical edge, projecting into the vertex $y\in \mathcal{H}_\gamma$, then we divide $f$ into $4$ parallel edges (because $y$ itself has been previously divided into $4$ quarter threads) and color each one of these new edges green independently with probability $r=1-(1-p)^{1/4}$.
    Now we declare $f$ open if at least one of these $4$ new edges is colored green.
    \item If $e=\{(v_i,z),(y,z)\}$ is a horizontal edge so that $\Xi(i)=1$ and $y$ is the thread that had one of its quarters assigned to $v_i$, then divide it into $2$ parallel horizontal edges.
    Color each one of these new edges red independently with probability $s=1-(1-p)^{1/2}$ and declare $e$ open if at least one of these new edges is colored red.
\end{enumerate}
Notice that, the resulting percolation process on $E\big((\gamma\cup\mathcal{H}_\gamma)\times\mathbb{Z}\big)$ is distributed as $\mathbb{P}^{\rho}_p( \cdot \mid \Delta_\gamma)$ (restricted to this set).

Before we construct the brochette percolation process on $\mathbb{Z}_+\times\mathbb{Z}$, let us recall that vertical edges that project into occupied threads in $\mathcal{H}_\gamma$ have been divided into four whereas horizontal edges having one endpoint that project into occupied threads in $\mathcal{H}_\gamma$ have been divided into two.
The reason for doing so is that we can now regard the handle-shaped detours around vertical edges $\{(v_i,z),(v_i,z+1)\}$ for which $\Xi(i)=1$ as illustrated in Figure \ref{f:1desenho}.

Now, for a vertical edge, $f'=\{(i,z),(i,z+1)\}$ with $\Xi(i)=0$ or a horizontal edge, $e'=\{(i,z),(i+1,z)\}$, declare it open if the edge $f=\{(v_i,z),(v_i,z+1)\}$ or respectively $e=\{(v_i,z),(v_{i+1},z)\}$, is open.
Now for a vertical edge $f'=\{(i,z),(i,z+1)\}$ with $\Xi(i)=1$ declare it open if either the corresponding edge $f=\{(v_i,z),(v_i,z+1)\}$ is open or if the handle-shaped detour around it has the bottom and top edges colored red and one of the $4$ vertical edges green.
This occurs with probability $q$.
One can check that this defines a process in $\mathbb{Z}_+\times \mathbb{Z}$ distributed as $P^{\mu_{u}\mathcal{B}_\gamma}_{p,q}$.

Moreover, we have coupled the processes in such a way that if $\big\{\text{$o\leftrightarrow \{n\}\times\mathbb{Z}$ in $[n]_0\times \mathbb{Z}$}\big\}$ occurs, then also $\big\{\text{$(v_0,0) \leftrightarrow \{v_n\}\times\mathbb{Z}$ in $(\gamma\cup\mathcal{H}_\gamma)\times\Z$}\big\}$ does.
Thus $\eqref{e:coupling}$ holds.

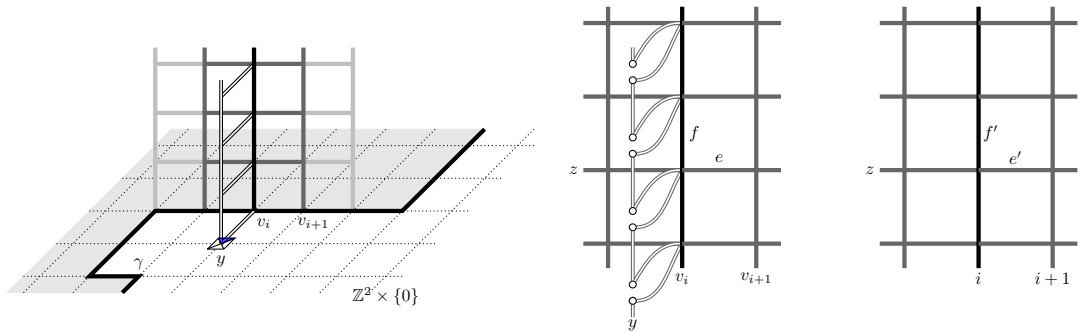
\begin{figure}[htb!]
\centering
\begin{tikzpicture}[scale=0.65, every node/.style={transform shape}]


\filldraw[black!10!](1/3,-1/3)--(3-1/3,0-1/3)--(3,0)--(2,0)--(4-2/3,2-2/3)--(9-2/3,2-2/3)--(10,3)--(4-1/3,3)--(1/3,-1/3);
\foreach \x in {1,2,3,4,5,6,7,8} \draw[densely dotted] (\x-1/3,-1/3)--(\x+3,3);
\foreach \x in {0,1,2,3,4} \draw[densely dotted] (\x/1.5+2/3,\x/1.5)--(\x/1.5+8+1/3,\x/1.5);
\filldraw[white](5-11/18,0+5/9)--(5-11/18+1/3,0+5/9)--(5-11/18+1/3+2/9,2/9+5/9)--(5-11/18+2/9,2/9+5/9)--(5-11/18,0+5/9);
\draw[black!25!, ultra thick] (4-2/3,4/3+1)--(5-2/3,4/3+1);
\draw[black!25!, ultra thick] (4-2/3,4/3+2)--(5-2/3,4/3+2);
\draw[black!25!, ultra thick] (4-2/3,4/3+3)--(5-2/3,4/3+3);
\draw[black!25!, ultra thick] (7-2/3,4/3+1)--(8-2/3,4/3+1);
\draw[black!25!, ultra thick] (7-2/3,4/3+2)--(8-2/3,4/3+2);
\draw[black!25!, ultra thick] (7-2/3,4/3+3)--(8-2/3,4/3+3);
\draw[black!25!, ultra thick] (4-2/3,4/3)--(4-2/3,5-1/3);
\draw[black!25!, ultra thick] (8-2/3,4/3)--(8-2/3,5-1/3);
\draw[black!60!, ultra thick] (5-2/3,4/3)--(5-2/3,5-1/3);
\draw[black!60!, ultra thick] (7-2/3,4/3)--(7-2/3,5-1/3);
\draw[black!60!, ultra thick] (5-2/3,4/3+1)--(7-2/3,4/3+1);
\draw[black!60!, ultra thick] (5-2/3,4/3+2)--(7-2/3,4/3+2);
\draw[black!60!, ultra thick] (5-2/3,4/3+3)--(7-2/3,4/3+3);
\draw[black, ultra thick] (3-1/3,0-1/3)--(3,0)--(2,0)--(4-2/3,2-2/3)--(9-2/3,2-2/3)--(10,3);
\draw[black, double](6-2/3,4/3)--(6-4/3,2/3);
\draw[black, double](6-2/3,7/3)--(6-4/3,5/3);
\draw[black, double](6-2/3,10/3)--(6-4/3,8/3);
\draw[black, double](6-2/3,13/3)--(6-4/3,11/3);
\draw[black, double](6-4/3,2/3)--(6-4/3,4);
\draw[black, ultra thick] (6-2/3,4/3)--(6-2/3,5-1/3);


\draw[black!60!, ultra thick](13.5-1,-1/3+0.5)--(13.5-1,5.5);
\draw[black!60!, ultra thick](15-1,2/3)--(13-1,2/3);
\draw[black!60!, ultra thick](15-1,2/3+1.5)--(13-1,2/3+1.5);
\draw[black!60!, ultra thick](15-1,2/3+3)--(13-1,2/3+3);
\draw[black!60!, ultra thick](15-1,2/3+4.5)--(13-1,2/3+4.5);
\draw[black!60!, double, rounded corners=1] (14-1,-0.5-1/3)--(14-1,-0.5) to  [out=0, in=240] (15-1,2/3) to  [out=180, in=60] (14-1,-0.5+1/3)--(14-1,1) to  [out=0, in=240](15-1,2/3+1.5) to  [out=180, in=60] (14-1,1+1/3)--(14-1,2.5) to  [out=0, in=240](15-1,2/3+3) to  [out=180, in=60] (14-1,2.5+1/3)--(14-1,4) to  [out=0, in=240](15-1,2/3+4.5) to  [out=180, in=60] (14-1,4+1/3)--(14-1,4+2/3);
\draw[black, ultra thick](15-1,-1/3+0.5)--(15-1,5.5);
\draw[black!60!, ultra thick](15-1,2/3)--(17-1,2/3);
\draw[black!60!, ultra thick](15-1,2/3+1.5)--(17-1,2/3+1.5);
\draw[black!60!, ultra thick](15-1,2/3+3)--(17-1,2/3+3);
\draw[black!60!, ultra thick](15-1,2/3+4.5)--(17-1,2/3+4.5);
\draw[black!60!, ultra thick](16.5-1,-1/3+0.5)--(16.5-1,5.5);


\filldraw[white] (13,-0.5) circle(2pt);
\draw (13,-0.5) circle(2pt);
\filldraw[white] (13,-0.5+1/3) circle(2pt);
\draw (13,-0.5+1/3) circle(2pt);
\filldraw[white] (13,-0.5) circle(2pt);
\draw (13,-0.5) circle(2pt);
\filldraw[white] (13,1) circle(2pt);
\draw (13,1) circle(2pt);
\filldraw[white] (13,1+1/3) circle(2pt);
\draw (13,1+1/3) circle(2pt);
\filldraw[white] (13,2.5) circle(2pt);
\draw (13,2.5) circle(2pt);
\filldraw[white]  (13,2.5+1/3) circle(2pt);
\draw  (13,2.5+1/3) circle(2pt);
\filldraw[white]  (13,4) circle(2pt);
\draw  (13,4) circle(2pt);
\filldraw[white] (13,4+1/3) circle(2pt);
\draw (13,4+1/3) circle(2pt);


\filldraw[blue, line width=0.1pt,fill opacity=0.6] (5-11/18+1/9+1/6,1/9+5/9)--(5-11/18+1/3+2/9,2/9+5/9)--(5-11/18+2/9,2/9+5/9)--(5-11/18+1/9+1/6,1/9+5/9);
\draw(5-11/18,0+5/9)--(5-11/18+1/3+2/9,2/9+5/9);
\draw(5-11/18+1/3,0+5/9)--(5-11/18+2/9,2/9+5/9);
\draw (5-11/18,0+5/9)--(5-11/18+1/3,0+5/9)--(5-11/18+1/3+2/9,2/9+5/9)--(5-11/18+2/9,2/9+5/9)--(5-11/18,0+5/9);
%
\node at (8,-0.4) {$\mathbb{Z}^2\times\{0\}$};
\node at (3,0.25) {$\gamma$};
\node at (5+1/3+0.2,4/3-0.25) {$v_i$};
\node at (6+1/3+0.2,4/3-0.25) {$v_{i+1}$};
\node at (5-1/3,2/3-0.35) {$y$};
\node at (14,-0.05) {$v_i$};
\node at (13,-1) {$y$};
\node at (15.5,-0.05){$v_{i+1}$};
\node at (20,-0.05) {$i$};
\node at (21.5,-0.05) {$i+1$};
\node at (11.8,2/3+1.5) {$z$};
\node at (17.8,2/3+1.5) {$z$};
\node at (20.25,2/3+1.5+0.75) {$f'$};
\node at (14.25,2/3+1.5+0.75) {$f$};
\node at (14.75,2/3+1.5+0.25) {$e$};
\node at (20.75,2/3+1.5+0.25) {$e'$};

\draw[black!60!, ultra thick](18.5,-1/3+0.5)--(18.5,5.5);
\draw[black!60!, ultra thick](20,2/3)--(18,2/3);
\draw[black!60!, ultra thick](20,2/3+1.5)--(18,2/3+1.5);
\draw[black!60!, ultra thick](20,2/3+3)--(18,2/3+3);
\draw[black!60!, ultra thick](20,2/3+4.5)--(18,2/3+4.5);
\draw[black, ultra thick](20,-1/3+0.5)--(20,5.5);
\draw[black!60!, ultra thick](20,2/3)--(22,2/3);
\draw[black!60!, ultra thick](20,2/3+1.5)--(22,2/3+1.5);
\draw[black!60!, ultra thick](20,2/3+3)--(22,2/3+3);
\draw[black!60!, ultra thick](20,2/3+4.5)--(22,2/3+4.5);
\draw[black!60!, ultra thick](21.5,-1/3+0.5)--(21.5,5.5);
\end{tikzpicture}
\caption{In the left we represent the path $\gamma$ along with one of its sites $v_i$ and the thread $y$ whose quarter thread is assigned to $v_i$ and is colored blue.
In the middle we illustrate the division of the horizontal edges in $E\big((\gamma \cup \mathcal{H}_\gamma)\times \mathbb{Z}\big)$ of the type $\{(v_i,z),(y,z)\}$ into two parallel edges and the resulting handle-shaped detours.
In the right we represent edges in $E([n]_0 \times \mathbb{Z})$.
Horizontal edges $e'$ are open if the corresponding $e$ are open.
Vertical edges $f'$ along the $i$-th column are open if either the corresponding edge $f$ is open or the handle-shaped detour around it has been colored green and red.}
\label{f:1desenho} 
\end{figure}

We are now ready to put together all the ingredients needed in order to prove our main result.
\begin{proof}[Proof of Theorem \ref{t:main}]
Let $\varepsilon=(1-1/\sqrt{2})^2(1-1/\sqrt[4]{2})/4$ so that $q=1/2+2\varepsilon$, when $p=1/2$.
Define also $u_c=1-(1-\rho_c)^{1/4}$. 
For such $\varepsilon>0$ and $u_c>0$, let $\delta=\delta(u_c,\varepsilon)>0$ and $\alpha=\alpha(u_c,\varepsilon)$ be given as in Theorem \ref{t:brochette_new}.
By \eqref{e:u_q_p}, $q$ varies continuously as a function of $p$, therefore we can chose $0<\delta'\leq\delta$ so that $q\geq 1/2+\varepsilon$ when $p=1/2-\delta'$.

For each $k\in\mathbb{Z}_+$ and $\gamma=(v_0,\cdots, v_n)\in \Phi_k(\Delta_k)$, it is clear that $n \geq 2^k L$.
Denote by $\mathcal{C}$ the largest cluster that touches the segment  $\{0\}\times[0, L]\times\{0\}$ and note that $v_0\in \{0\}\times[0, L]\times\{0\}$.
	Then we have
	\begin{eqnarray}\mathbb{P}^\rho_p(|\mathcal{C}|>2^k L)&\geq& \frac{1}{2}\;\; \mathbb{P}^\rho_p(|\mathcal{C}|>2^k L| \Delta_k)\nonumber\\
	&=& \frac{1}{2}\sum_{\gamma\in \Phi_k(\Delta_k)}\mathbb{P}^\rho_p(|\mathcal{C}|>2^k L| \Delta_\gamma)\nu_\rho(\Delta_\gamma|\Delta_k)\nonumber\\
	&\geq& \frac{1}{2}\sum_{\gamma\in \Phi_k(\Delta_k)}\mathbb{P}^\rho_p(\text{$(v_0,0) \leftrightarrow \{v_n\}\times\mathbb{Z}$ in $(\gamma\cup\mathcal{H}_\gamma)\times\Z$}| \Delta_\gamma)\nu_\rho(\Delta_\gamma|\Delta_k)\nonumber\\
	&\stackrel{(\ref{e:coupling})}{\geq} &\frac{1}{2}\sum_{\gamma\in \Phi_k(\Delta_k)}  {P}^{\mu_{u,\mathcal{B}_\gamma}}_{p,q}\big(\text{$o\leftrightarrow \{n\}\times\mathbb{Z}$ in $[n]_0\times \mathbb{Z}$}\big)\nu_\rho(\Delta_\gamma|\Delta_k)\stackrel{\eqref{e:brochette_new_n}}{\geq} \frac{\alpha}{2}.\nonumber \end{eqnarray}
	
Since the lower bound above holds for every $k\in\Z_+$, this proves that $p_c(\rho)\leq \frac{1}{2}-\delta'$ uniformly for $\rho>\rho_c$ concluding the proof of Theorem \ref{t:main}.
\end{proof}

\section{Concluding Remarks}
\label{s:conc_rema}

In this work we have managed to compare the critical points of percolation on a dilute cubic lattice with columnar disorder in $\mathbb{Z}^3$ with the critical point of bond percolation in $\mathbb{Z}^2$ uniformly in the disorder intensity.
Of course, one would expect to show that the critical curve $p_c(\rho)$ is non-decreasing throughout the interval $(\rho_c,1]$.
This seems to be an interesting and hard question.

Indeed, our argument goes like this: right above $\rho=\rho_c$ we can find copies of $\mathbb{Z}^2$ embedded into $\mathbb{Z}^3$ and independent percolation at $p=1/2$ restricted to one of these copies is critical.
By considering only the effect of the unit length threads, we actually have something substantially better for percolation these embedded copies of $\mathbb{Z}^2$.
This last comparison, although intuitive is not at all trivial and relies on a coupling with brochette percolation.

Now assume that we have two densities $\rho' > \rho$ both in the interval $[\rho_c(\mathbb{Z}^2),1]$ and consider the respective 2-$d$ percolation processes in $\mathbb{Z}^2$ coupled in the usual monotone way.
The removal of columns with density $1-\rho$ will of course leave a structure that is thinner than that with density $1-\rho'$.
It seem reasonable to state that the latter is strictly better for bond percolation then the former.
However, the available techniques of differential inequalities do not seem to work in this context where enhancements are performed at random in the presence of correlations that do not decay with distance.

In our view a better understanding on how the presence of lower-dimensional disorder affects the critical point is an interesting question to be addressed.

There is however another value of $\rho$ at which we can say something related to the Figure \ref{f:critical_curve}, namely the other extreme of the interval $[\rho_c,1]$.
It is not hard to prove that the critical curve is continuous as $\rho$ approaches $1$.
Here it follows a sketch of the  argument which requires knowledge of the Grimmett and Marstrand arguments \cite{grimmett90}.
Consider the set of all paths in $\mathbb{Z}^2$ starting from the origin that are directed, that is, that only take up and right steps.
For each such path $\gamma$ let $F_\gamma=\gamma\times\mathbb{Z} \subset \mathbb{Z}^3$ be the set of sites in $\mathbb{Z}^3$ that project into $\gamma$.
Theorem \cite[Theorem A, page 447]{grimmett90} states that there exists a positive integer $k=k(\varepsilon)$ such that $p_c^{\text{site}}(2kF_\gamma + B(k)) < p_c^{\text{site}}(\mathbb{Z}^3)+\varepsilon$.
Although it does not follow from their statement, an inspection of the proof reveals that $k$ can be taken uniformly in $\gamma$.
That is, thickening $F_\gamma$ by stretching $2k$ times and then filling with boxes of radius $k$ results in a `zigzagging slab' whose critical point falls below $p_c^{\text{site}}(\mathbb{Z}^3)+\varepsilon$ for every $\gamma$.
Fixed $\varepsilon$ and the corresponding $k=k(\varepsilon)$ we now pick $\rho \in (0,1)$ such that $\rho^{(2k+1)^2}$ exceeds the critical threshold for oriented percolation on $\mathbb{Z}^2$.
Therefore, for such a value of $\rho$ one can find, with positive probability, an infinite directed path of adjacent boxes of side length $2k$ whose sites are all occupied.
Conditional on such a path, the structure of $\mathbb{Z}^3$ that projects to it is zigzag slab as above therefore, $p_c(\rho)<p_c^{\text{site}}(\mathbb{Z}^3)+\varepsilon$ which concludes the argument.

Our result leads naturally to the question whether an upper bound strictly smaller than $1/2$ still holds if the environment $\Lambda$ is distributed like the incipient infinite cluster in $\mathbb{Z}^2$.
We believe that this answer is positive but, unfortunately we fall short of proving so and leave it as an interesting open question.

\vspace{1cm}

\begin{tabular}{cl}
 \hspace{-1cm}$^*$   &\hspace{-0.9cm}  \textsc{Universidade Federal de Minas Gerais (UFMG),}\\
  &\hspace{-0.9cm}  \textsc{Dep. de Matemática, 31270-901 Belo Horizonte, MG - Brazil.}\\
  &\\
  \hspace{-1cm}$^\dagger$   &\hspace{-0.9cm} \textsc{Instituto de Matemática Pura e Aplicada(IMPA)}\\
  &\hspace{-0.9cm} \textsc{Estrada Dona Castorina 110, 22460-320 Rio de Janeiro, RJ - Brazil.}
\end{tabular}

\vspace{0.5cm}

\begin{tabular}{ll}
\hspace{-0.48cm} E-mails:    & \texttt{mhilario@mat.ufmg.br},  \\
     & \texttt{marcospy6@ufmg.br}, \\
     & \texttt{rsanchis@mat.ufmg.br}.
\end{tabular}


\begin{thebibliography}{9}
\bibitem{aizenman87} Aizenman, M., and Barsky, D. J. (1987). Sharpness of the phase transition in percolation models. \textit{Communications in Mathematical Physics}, 108(3), 489-526.

\bibitem{aizenman87_2} Aizenman, M., Chayes,  J. T., Chayes, L., Newman C. M. (1987).
The phase boundary in dilute and random Ising and Potts ferromagnets. \textit{Journal of Physics A}, 20(5), L313-L318.

\bibitem{aiznman83} Aizenman, M., Chayes, J., Chayes, L., Fr\"ohlich, J., Russo R. (1983). On a Sharp Transition from Area Law to Perimeter Law in a System of Random Surfaces {\it Communications in  Mathematical Physics} 92, 19-69.

\bibitem{aizenman91} Aizenman, M., Grimmett, G. (1991). Strict monotonicity for critical points in percolation and ferromagnetic models. \textit{Journal of Statistical Physics}, 63(5-6), 817-835.

\bibitem{bramson91} Bramson, M., Durrett, R., Schonmann, R. H. (1991). The contact processes in a random environment. \textit{The Annals of Probability}, 960-983.

\bibitem{campanino91} Campanino, M., Klein, A. (1991). Decay of two-point functions for (d+1)-dimensional percolation, Ising and Potts models with $d$-dimensional disorder.
\textit{Communications in Mathematical Physics} 135, 483-497.

\bibitem{campanino91_2} Campanino, M., Klein, A. Perez, J.F. (1991).
Localization in the ground state of the Ising model with a random transverse field. \textit{Communications in Mathematical Physics} 135(3), 499-515.

\bibitem{chayes98} Chayes, J.,  Puha, A.,  Sweet, T. (1998).  Independent and dependent percolation {\it In: IAS/Park City Mathematics Series} Vol.6, AMS, Providence, RI.

\bibitem{chayes00}
Chayes, L.; Schonmann, R.H. (2000). Mixed percolation as a bridge between site and bond percolation. 
\textit{Ann. Appl. Probab.} 10, no. 4, 1182--1196.

\bibitem{duminil-copin18}
Duminil-Copin, H., Hilario, M. R., Kozma, G.,  Sidoravicius, V. (2018). Brochette percolation. \textit{Israel Journal of Mathematics} Vol. 225 (1), pp. 479-501.

\bibitem{georgii81}
Georgii, H. (1981). Spontaneous magnetization of randomly dilute ferromagnets. \textit{J. Stat. Phys.} 25, 369–396.

\bibitem{georgii84}
Georgii, H. (1984). On the ferromagnetic and the percolative region of random spin systems. \textit{Advances in Applied Probability}, 16(4), 732-765.

\bibitem{grassberger17_2}
Grassberger, P. (2017). Universality and asymptotic scaling in drilling percolation \textit{Phys. Rev. E}, 95 (1), 10103-10107.

\bibitem{grassberger17}
Grassberger, P., Hil\'ario, M. R., Sidoravicius, V. (2017). Percolation in Media with Columnar Disorder. \textit{Journal of Statistical Physics}, 168(4), 731-745.

\bibitem{griffiths68}
Griffiths, R. B.  and Lebowitz, J. L. (1968). Random Spin Systems: Some Rigorous Results. \textit{Journal of Mathematical Physics}, 9(8), 1284-1292.

\bibitem{grimmett90}
Geoffrey Richard Grimmett, G. R., J. M. Marstrand J. M. (1990). The supercritical phase of percolation is well-behaved. \textit{Proc. R. Soc. Lond.} 430, 439-457.

\bibitem{grimmett99} Grimmett, G. (1999). Percolation. \textit{Springer-Verlag Berlin Heidelber} 2nd.~edition.

\bibitem{harris60}
Harris, T. E. (1960). A lower bound for the critical probability in a certain percolation process. \textit{In Mathematical Proceedings of the Cambridge Philosophical Society} (Vol. 56, No. 1, pp. 13-20). Cambridge University Press.

\bibitem{hilario19}
Hil\'ario, M. R., Sidoravicius, V. (2019). Bernoulli line percolation. \textit{Stochastic Processes and Application}  Vol. 129 No. 12 pp.5037-5072.

\bibitem{hilario19_2}
Hilario, M. R, Sá, M., Sanchis R., Teixeira, A. (2019).
Phase transition for percolation on randomly stretched lattice.
\textit{arXiv preprint} arXiv:1912.03320.

\bibitem{hoffman05}
Hoffman, C. (2005). Phase transition in dependent percolation. \textit{Communications in mathematical physics}, 254(1), 1-22.

\bibitem{kantor86}
Kantor, Y. (1986). Three-dimensional percolation with removed lines of sites. \textit{Physical Review B}, 33(5), 3522.

\bibitem{kesten80}
Kesten, H. (1980). The critical probability of bond percolation on the square lattice equals 1/2. \textit{Communications in mathematical physics}, 74(1), 41-59.

\bibitem{kesten12}
Kesten, H., Sidoravicius, V., Vares, M. E. (2012). Oriented percolation in a random environment. \textit{arXiv preprint} arXiv:1207.3168.

\bibitem{madras94}
Madras, N., Schinazi, R.,  Schonmann, R. H. (1994). On the critical behavior of the contact process in deterministic inhomogeneous environments.  \textit{The Ann.~Probab.}, 1140-1159.

\bibitem{mccoy68}
McCoy,~B.M., Wu,~T. T. (1965). Ising Model with Random Impurities. I. Thermodynamics. \textit{Physical Review}, 176(2), 631-643.

\bibitem{menshikov86}
Menshikov, M. V. (1986). Coincidence of critical points in percolation problems. \textit{In Soviet Mathematics Doklady} (Vol. 33, pp. 856-859).

\bibitem{menshikov87}
Menshikov, M. V. (1987). Quantitative Estimates and Rigorous Inequalities for Critical Points of a Graph and Its Subgraphs  \textit{Theory Probab. Appl.}, 32(3), 544–547.

\bibitem{newman96}
Newman, C.M., Volchan, S. B. (1996). Persistent survival of one-dimensional contact processes in random environments. \textit{The Annals of Probability} 24(1), 411-421.

\bibitem{pete08}
Pete, G. (2008). Corner percolation on $\mathbb{Z}^2$ and the square root of 17. \textit{Ann.~Probab.} 36(5), 1711-1747.


\bibitem{schrenk16}
Schrenk, K. J., Hil\'ario, M. R.  Sidoravicius, V., Ara\'ujo, N. A. M., Herrmann, H. J., Thielmann, M., Teixeira, A. (2016). Critical Fragmentation Properties of Random Drilling: How Many Holes Need to Be Drilled to Collapse a Wooden Cube? \textit{Physical Review Letters}, 116 (5) 055701.


\bibitem{winkler00}
Winkler, P. (2000). Dependent percolation and colliding random walks. \textit{Random Struct. Alg.}, 16, 58-84. 

\bibitem{zhang94}
Zhang, Y. (1994). A note on inhomogeneous percolation. \textit{The Annals of Probability}, 803-819.
\end{thebibliography}
\end{document}